\documentclass[12pt]{amsart}
\usepackage{epsfig,amsthm,amsmath,amsfonts,amssymb,nicefrac}
\makeatletter 
\usepackage{graphicx,color,amsmath,amssymb,amsthm}

\theoremstyle{plain}             
\newtheorem{theorem}{Theorem}[section]
\newtheorem{lemma}[theorem]{Lemma}

\newtheorem{remark}[theorem]{Remark}
 
\newcommand{\mix}{\operatorname{mix}}
\newcommand{\SVD}{\operatorname{SVD}}
\newcommand{\TF}{\operatorname{TF}}
\newcommand{\TT}{\operatorname{TT}}

\newcommand{\trace}{\operatorname{trace}}
\renewcommand{\d}{\operatorname{d}\!}

\begin{document}
\title{Analysis of tensor approximation schemes for continuous functions}
\author{Michael Griebel}
\address{Michael Griebel,
Institut f\"ur Numerische Simulation, Universit\"at Bonn, Endenicher Allee 19b, 53115 Bonn,
and Fraunhofer Institute for Algorithms and Scientific Computing (SCAI), Schloss Birlinghoven, 
53754 Sankt Augustin, Germany}
\email{griebel@ins.uni-bonn.de}
\author{Helmut Harbrecht}
\address{Helmut Harbrecht,
Departement Mathematik und Informatik, Universit\"at Basel, Spiegelgasse 1, 4051 Basel, Switzerland}
\email{helmut.harbrecht@unibas.ch}
\keywords{Tensor format, approximation error, rank complexity, 
Sobolev space with dimension weights}

\begin{abstract}
In this article, we analyze tensor approximation 
schemes for continuous functions. We assume that the 
function to be approximated lies in an isotropic Sobolev 
space and discuss the cost when approximating this 
function in the continuous analogue of the Tucker tensor 
format or of the tensor train format. We especially show 
that the cost of both approximations are dimension-robust 
when the Sobolev space under consideration provides 
appropriate dimension weights.
\end{abstract}

\maketitle

\section{Introduction}
The efficient approximate representation of multivariate functions 
is an important task in numerical analysis and scientific computing.
In this article, we hence consider the approximation of functions 
which live on the product of $m$ bounded domains $\Omega_1\times\dots\times
\Omega_m$, each of which satisfies $\Omega_j\subset\mathbb{R}^{n_j}$. 
Besides a sparse grid approximation of the function under 
consideration, being discussed in, e.g., \cite{BG,GH13a,GH13b,Z}, 
one can also apply a low-rank approximation by means of a tensor 
approximation scheme, see, e.g., \cite{GRA,Hack,HaKu,O,OT} 
and the references therein.

The low-rank approximation in the situation of the product of $m=2$ 
domains is well understood. It is related to the \emph{singular value 
decomposition\/} and has been studied for arbitrary product domains
in, e.g., \cite{GH14,GH18}, see also \cite{T1,T2,T3} for the periodic
case. However, the situation is not that clear for the product of $m>2$ 
domains, where one ends up with \emph{tensor decompositions\/}. 
Such tensor decompositions are generalizations of the well known singular 
value decomposition and the corresponding low-rank matrix approximation 
methods of two dimensions to the higher-dimensional setting. There, 
besides the curse of dimension, we encounter -- due to the non-existence
of an Eckart-Young-Mirsky theorem -- that the concepts of singular 
value decomposition and low-rank approximation can be generalized 
to higher dimensions in more than one way. Consequently, there exist 
many generalizations of the singular value decomposition of a function and 
of low-rank approximations to tensors. To this end, various schemes have 
been developed over the years in different areas of the sciences and have 
successfully been applied to various high-dimensional problems ranging 
from quantum mechanics and physics via biology and econometrics, computer 
graphics and signal processing to numerical analysis. Recently, tensor 
methods have even been recognized as special deep neural networks in 
machine learning and big data analysis \cite{CohenShashua,KNO}. As tensor 
approximation schemes, we have, for example, matrix product states, DMRG, 
MERA, PEPS, CP, CANDECOMP, PARAFAC, Tucker, tensor train, tree tensor 
networks and hierarchical Tucker, to name a few. A mathematical introduction 
into tensor methods is given in the seminal book \cite{Hack}, while a survey 
on existing methods and their literature can be found in \cite{GKT}. Also 
various software packages have been developed for an algebra of operators 
dealing with tensors.

Tensor methods are usually analyzed as low-rank approximations to a 
full {\em discrete tensor\/} of data with respect to the $\ell_2$-norm or 
Frobenius-norm. In this respective, they can be seen as compression 
methods which may avoid the curse of dimensionality, which is inherent 
in the full tensor representation. The main tool for studying tensor 
compression schemes is the so-called {\em tensor-rank\/}, compare 
\cite{Falco1,Falco2,Hack}. Thus, instead of $\mathcal{O}(N^n)$ storage, 
as less as $\mathcal{O}(nNr^3)$ or even only $\mathcal{O}(nNr^2)$ 
storage is needed, where $N$ denotes the number of data points in one 
coordinate direction, $n$ denotes the dimension of the tensor under 
consideration and $r$ denotes the respective tensor rank of the data.
The cost complexity of the various algorithms working with sparse 
tensor representations is correspondingly reduced and working in 
a sparse tensor format allows to alleviate or to completely break 
the curse of dimension for suitable tensor data classes, i.e., for
sufficiently small $r$.

However, the question where the tensor data stem from and the issue 
of the accuracy of the full tensor approximation, i.e., the discretization
error of the full tensor itself and its relation to the error of a subsequent 
low-rank tensor approximation, is usually not adequately addressed.\footnote{We 
are only aware of \cite{BD1,BD2,M16,SU}, where this question has been
considered so far.} Instead, only the approximation property of a low-rank tensor 
scheme with respect to the full tensor data is considered. But the former 
question is important since it clearly makes no sense to derive a tensor 
approximation with an error that is substantially smaller than the error 
which is already inherent in the full tensor data due to some discretization 
process for a continuous high-dimensional function which stems from some 
certain function class.

The approximation rates to \emph{continuous functions\/} can be
determined by a recursive use of the singular value decomposition, 
which is successively applied to convert the function into a specific 
continuous tensor format. We studied the singular value decomposition 
for arbitrary domains in \cite{GH14,GH18} and we now can apply these
results to discuss approximation rates of continuous tensor formats. In 
the present article, given a function $f\in H^k(\Omega_1\times\dots
\times\Omega_m)$, we study the continuous analogues of the Tucker 
tensor decomposition and of the tensor train decomposition. We give 
bounds on the ranks required to ensure that the tensor decomposition 
admits a prescribed target accuracy. Especially, our 
analysis takes into account the influence of errors induced by 
truncating infinite expansions to finite ones. We therefore study
an \emph{algorithm\/} that computes the desired tensor expansion 
which is in contrast to the question of the smallest tensor-rank.
We finally show that (isotropic) Sobolev spaces
\emph{with dimension weights\/} help to beat the curse of 
dimension when the number $m$ of product domains tends 
to infinity.

Besides the simple situation of $\Omega_1 = \dots = \Omega_m 
= [0,1]$, which is usually considered in case of tensor decompositions, 
there are many more applications of our general setting. For 
example, non-Newtonian flow can be modeled by a coupled 
system which consists of the Navier Stokes equation for the flow 
in a three-dimensional geometry described by $\Omega_1$ and of 
the Fokker-Planck equation in a $3(\ell-1)$-dimensional configuration 
space $\Omega_2\times\dots\times\Omega_{\ell}$, consisting of
$\ell-1$ spheres. Here, $\ell$ denotes the number of atoms in a 
chain-like molecule which constitutes the non-Newtonian 
behavior of the flow, for details see \cite{BKS,LL,LOZ,RG}. 
Another example is homogenization. After unfolding \cite{CDG}, 
a two-scale homogenization problem gives raise to the product of 
the macroscopic physical domain and the periodic microscopic 
domain of the cell problem, see \cite{Mat}. For multiple scales,
several periodic microscopic domains appear which reflect the 
different separable scales, see e.g.\ \cite{HS}. 
Also the $m$-th moment of linear elliptic boundary value 
problems with random source terms, i.e.~$Au(\omega)=f(\omega)$ 
in $\Omega$, are known to satisfy a deterministic partial differential 
equation on the $m$-fold product domain $\Omega\times\dots\times
\Omega$. There, the solution's $m$-th moment $\mathcal{M}_u$
is given by the equation 
\[
  (A\otimes\dots\otimes A)\mathcal{M}_u = \mathcal{M}_f\ 
  	\text{in}\ \Omega\times\dots\times\Omega,
\]
see \cite{ST1,ST3}. This approach extends to boundary value 
problems with random diffusion and to random domains as well 
\cite{CS,HSS08}. Moreover, we find the product of several domains 
in quantum mechanics for e.g.~the Schr\"odinger equation or the 
Langevin equation, where each domain is three-dimensional and 
corresponds to a single particle. Finally, we encounter it in uncertainty 
quantification, where one has the product of the physical domain 
$\Omega_1$ and of in general infinitely many intervals $\Omega_2 
= \Omega_3 = \Omega_4 = \dots$ for the random input parameter, 
which reflects its series expansion by the Karhunen-L\`oeve 
decomposition or the L\'evy-Ciesielski decomposition.

The remainder of this article is organized as follows:
In Section \ref{sec:SVD}, we give a short introduction 
to our results on the singular value decomposition, which 
are needed to derive the estimates for the continuous tensor
decompositions. Then, in Section \ref{sec:TF}, we study 
the continuous Tucker tensor format, computed by means of 
the higher-oder singular value decomposition. Next, we study 
the continuous tensor train decomposition in Section \ref{sec:TT},
computed by means of a repeated use of a vector-valued 
singular value decomposition. Finally, Section\ \ref{sec:conrem} 
concludes with some final remarks.

Throughout this article, to avoid the 
repeated use of generic but unspecified constants, 
we denote by $C \lesssim D$ that $C$ is bounded 
by a multiple of $D$ independently of parameters 
which $C$ and $D$ may depend on. Obviously, 
$C \gtrsim D$ is defined as $D \lesssim C$, and 
$C \sim D$ as $C \lesssim D$ and $C \gtrsim D$.
Moreover, given a Lipschitz-smooth domain 
$\Omega\subset\mathbb{R}^n$, $L^2(\Omega)$ means the 
space of square integrable functions on $\Omega$. For real 
numbers $k\ge 0$, the associated Sobolev space is denoted by
$H^k(\Omega)$, where its norm $\|\cdot\|_{H^k(\Omega)}$ 
is defined in the standard way, compare \cite{MCL,STEIN}. 
As usual, we have $H^0(\Omega) = L^2(\Omega)$. The 
seminorm in $H^k(\Omega)$ is denoted by $|\cdot|_{H^k(\Omega_1)}$. 
Although not explicitly written, our subsequent analysis covers also the 
situation of $\Omega$ being not a domain but a (smooth) 
manifold.

\section{Singular value decomposition}\label{sec:SVD}
\subsection{Definition and calculation}
Let $\Omega_1\subset\mathbb{R}^{n_1}$ and $\Omega_2$
be Lipschitz-smooth domains. To represent functions 
$f\in L^2(\Omega_1\times\Omega_2)$ on the tensor 
product domain $\Omega_1\times\Omega_2$ in an efficient 
way, we will consider low-rank approximations which separate the 
variables $\boldsymbol{x}\in\Omega_1$ and $\boldsymbol{y}\in\Omega_2$ in 
accordance with
\begin{equation}\label{eq:low rank}
   f(\boldsymbol{x},\boldsymbol{y})\approx f_r(\boldsymbol{x},\boldsymbol{y}) 
   	:= \sum_{\alpha=1}^r \sqrt{\lambda(\alpha)}\varphi(\boldsymbol{x},\alpha)\psi(\alpha,\boldsymbol{y}).
\end{equation}
It is well known (see e.g.~\cite{L,SCH,Simsa}) that the best 
possible representation \eqref{eq:low rank} in the $L^2$-sense is 
given by the singular value decomposition, also called Karhunen-L\`oeve 
expansion.\footnote{We refer the reader to \cite{ST} for a comprehensive 
historical overview on the singular value decomposition.} Then, the 
coefficients $\sqrt{\lambda(\alpha)}\in\mathbb{R}$ are the singular values 
and the $\varphi(\alpha)\in L^2(\Omega_1)$ and $\psi(\alpha)\in L^2(\Omega_2)$ 
are the left and right ($L^2$-normalized) eigenfunctions of the integral operator
\[
  \mathcal{S}_f:L^2(\Omega_1)\to L^2(\Omega_2),
  \quad u\mapsto (\mathcal{S}_fu)(\boldsymbol{y}) := 
  	\int_{\Omega_1} f(\boldsymbol{x},\boldsymbol{y})u(\boldsymbol{x})\d\boldsymbol{x}.
\]
This means that
\begin{equation}\label{eq:LREV}
  \sqrt{\lambda(\alpha)}\psi(\alpha,\boldsymbol{y}) 
  	= \big(\mathcal{S}_f\varphi(\alpha)\big)(\boldsymbol{y})
		\quad\text{and}\quad
  \sqrt{\lambda(\alpha)}\varphi(\boldsymbol{x},\alpha) 
  	= \big(\mathcal{S}_f^\star\psi(\alpha)\big)(\boldsymbol{x}),
\end{equation}
where
\[
  \mathcal{S}_f^\star:L^2(\Omega_2)\to L^2(\Omega_1),
  \quad v\mapsto (\mathcal{S}_f^\star v)(\boldsymbol{x}) 
  := \int_{\Omega_2} f(\boldsymbol{x},\boldsymbol{y})v(\boldsymbol{y})\d\boldsymbol{y}.
\]
is the adjoint of $\mathcal{S}_f$. Especially, the left and right eigenfunctions
$\{\varphi(\alpha)\}_{\alpha=1}^\infty$ and $\{\psi(\alpha)\}_{\alpha=1}^\infty$ 
form orthonormal bases in $L^2(\Omega_1)$ and $L^2(\Omega_2)$, respectively.

In order to compute the singular value decomposition, we need to
solve the eigenvalue problem 
\[
  \mathcal{K}_f\varphi(\alpha) = \lambda(\alpha)\varphi(\alpha)
\]
for the integral operator
\begin{equation}\label{kern}
  \mathcal{K}_f=\mathcal{S}_f^\star\mathcal{S}_f:L^2(\Omega_1)\to L^2(\Omega_1),\quad 
  	u\mapsto (\mathcal{K}_fu)(\boldsymbol{x}) := \int_{\Omega_1} k_f(\boldsymbol{x},\boldsymbol{x}')u(\boldsymbol{x}')\d\boldsymbol{x}'.
\end{equation}
Since $f\in L^2(\Omega_1\times\Omega_2)$, the kernel
\begin{equation}	\label{eq:kernel}
  k_f(\boldsymbol{x},\boldsymbol{x}') = \int_{\Omega_2} 
  	f(\boldsymbol{x},\boldsymbol{y}) f(\boldsymbol{x}',\boldsymbol{y}) \d\boldsymbol{y} 	
  		\in L^2(\Omega_1\times\Omega_1).
\end{equation}
is a symmetric Hilbert-Schmidt kernel. Hence, there exist 
countably many eigenvalues
\[
\lambda(1)\ge\lambda(2)\ge\cdots\ge\lambda(m)\overset{m\to\infty}{\longrightarrow} 0
\]
and the associated eigenfunctions $\{\varphi(\alpha)\}_{\alpha\in\mathbb{N}}$ 
constitute an orthonormal basis in $L^2(\Omega_1)$. 

Likewise, to obtain an orthonormal basis of
$L^2(\Omega_2)$, we can solve the eigenvalue problem 
\[
  \widetilde{\mathcal{K}}_f\psi(\alpha) = \widetilde{\lambda}(\alpha)\psi(\alpha)
\]
for the integral operator
\[
  \widetilde{\mathcal{K}}_f=\mathcal{S}_f\mathcal{S}_f^\star:
  	L^2(\Omega_2)\to L^2(\Omega_2),\quad 
		u\mapsto (\widetilde{\mathcal{K}}_fu)(\boldsymbol{y}) := \int_{\Omega_2} 
			\widetilde{k}_f(\boldsymbol{y},\boldsymbol{y}')u(\boldsymbol{y}')\d\boldsymbol{y}'
\]
with symmetric Hilbert-Schmidt kernel
\begin{equation}\label{eq:tildekernel}
  \widetilde{k}_f(\boldsymbol{y},\boldsymbol{y}') = \int_{\Omega_1} 
  	f(\boldsymbol{x},\boldsymbol{y}) f(\boldsymbol{x},\boldsymbol{y}') \d\boldsymbol{x} 	
  		\in L^2(\Omega_2\times\Omega_2).
\end{equation}
It holds $\lambda(\alpha) = \widetilde{\lambda}(\alpha)$ and the 
sequences $\{\varphi(\alpha)\}$ and $\{\psi(\alpha)\}$ are related by
\eqref{eq:LREV}.

\subsection{Regularity of the eigenfunctions}\label{subsec:regularity}
Now, we consider functions $f\in H^k(\Omega_1\times\Omega_2)$.
In the following, we collect results from \cite{GH14,GH18} concerning 
the singular value decomposition of such functions. We repeat the proof 
whenever needed for having explicit constants. To this end, we define the 
mixed Sobolev space $H_{\mix}^{k,\ell}(\Omega_1\times\Omega_2)$ as
a tensor product of Hilbert spaces
\[
  H_{\mix}^{k,\ell}(\Omega_1\times\Omega_2)
  	:= H^k(\Omega_1)\otimes H^\ell(\Omega_2),
\]
which we equip with the usual cross norm
\[
  \|f\|_{H_{\mix}^{k,\ell}(\Omega_1\times\Omega_2)}
  	:= \sqrt{\sum_{|\boldsymbol\alpha|\le k}\sum_{|\boldsymbol\beta|\le\ell}
		\bigg\|\frac{\partial^{|\boldsymbol\alpha|}\partial^{|\boldsymbol\beta|}}
			{\partial{\bf x}^{\boldsymbol\alpha}\partial{\bf y}^{\boldsymbol\beta}}
			f\bigg\|_{L^2(\Omega_1\times\Omega_2)}^2}.
\]
Note that
\[
  H^k(\Omega_1\times\Omega_2)\subset
  H_{\mix}^{k,0}(\Omega_1\times\Omega_2),\quad
  H^k(\Omega_1\times\Omega_2)\subset
  H_{\mix}^{0,k}(\Omega_1\times\Omega_2).
\]

\begin{lemma}		\label{lem:continuity}
Assume that $f\in H^k(\Omega_1\times\Omega_2)$
for some fixed $k\ge 0$. Then, the operators
\[
  \mathcal{S}_f:L^2(\Omega_1)\to H^k(\Omega_2),\quad
  \mathcal{S}_f^\star:L^2(\Omega_2)\to H^k(\Omega_1)
\]
are continuous with 
\[
\big\|\mathcal{S}_f\big\|_{L^2(\Omega_1)\to H^k(\Omega_2)}
\le\|f\|_{H_{\mix}^{0,k}(\Omega_1\times\Omega_2)},\quad
\big\|\mathcal{S}_f^\star\big\|_{L^2(\Omega_2)\to H^k(\Omega_1)}
\le\|f\|_{H_{\mix}^{k,0}(\Omega_1\times\Omega_2)}.
\]
\end{lemma}

\begin{proof}
From $H^k(\Omega_1\times \Omega_2)\subset
H_{\mix}^{0,k}(\Omega_1\times \Omega_2)$ it 
follows for $f\in H^k(\Omega_1\times \Omega_2)$ that 
$f\in H_{\mix}^{0,k}(\Omega_1\times \Omega_2)$. 
Therefore, the operator $\mathcal{S}_f:L^2(\Omega_1)
\to H^k(\Omega_2)$ is continuous since
\begin{align*}
  \big\|\mathcal{S}_fu\big\|_{H^k(\Omega_2)}
  	&= \sup_{\|v\|_{H^{-k}(\Omega_2)}=1}
		(\mathcal{S}_fu,v)_{L_2(\Omega_2)}\\
	&= \sup_{\|v\|_{H^{-k}(\Omega_2)}=1}(f,u\otimes v)_{L^2(\Omega_1\times\Omega_2)}\\
	&\le\sup_{\|v\|_{H^{-k}(\Omega_2)}=1}\|f\|_{H_{\mix}^{0,k}(\Omega_1\times\Omega_2)}
		\|u\otimes v\|_{H_{\mix}^{0,-k}(\Omega_1\times\Omega_2)}\\
	&\le\|f\|_{H_{\mix}^{0,k}(\Omega_1\times\Omega_2)}\|u\|_{L^2(\Omega_1)}.
\end{align*}
Note that we have used here
$H_{\mix}^{0,-k}(\Omega_1\times\Omega_2) =
L^2(\Omega_1)\otimes H^{-k}(\Omega_2)$.
Proceeding likewise for $\mathcal{S}_f^\star:
L^2(\Omega_2)\to H^k(\Omega_1)$ completes the
proof.
\end{proof}

\begin{lemma}\label{lem:regularity1}
Assume that $f\in H^k(\Omega_1\times\Omega_2)$ for some fixed 
$k\ge 0$. Then, it holds $\mathcal{S}_f\varphi(\alpha)\in H^k(\Omega_2)$ and
$\mathcal{S}_f^\star\psi(\alpha)\in H^k(\Omega_1)$ for all $\alpha\in\mathbb{N}$
with
\[
 \|\varphi(\alpha)\|_{H^k(\Omega_1)}\le\frac{1}{\sqrt{\lambda(\alpha)}}
 	\|f\|_{H_{\mix}^{k,0}(\Omega_1\times\Omega_2)},
 \quad
 \|\psi(\alpha)\|_{H^k(\Omega_2)}\le\frac{1}{\sqrt{\lambda(\alpha)}}
 	\|f\|_{H_{\mix}^{0,k}(\Omega_1\times\Omega_2)}.
\]
\end{lemma}

\begin{proof}
According to \eqref{eq:LREV} and Lemma~\ref{lem:continuity}, 
we have
\[
\|\varphi(\alpha)\|_{H^k(\Omega_1)} = \frac{1}{\sqrt{\lambda(\alpha)}}
	\big\|\mathcal{S}_f^\star\psi(\alpha)\big\|_{H^k(\Omega_1)}\le\frac{1}{\sqrt{\lambda(\alpha)}}
	\|f\|_{H_{\mix}^{k,0}(\Omega_1\times\Omega_2)}\|\psi(\alpha)\|_{L^2(\Omega_2)}.
\]
This proves the first assertion. The second assertion follows
by duality. 
\end{proof}

As an immediate consequence of Lemma~\ref{lem:regularity1}, 
we obtain 
\[
	\sum_{\alpha=1}^r\lambda(\alpha)\|\varphi(\alpha)\|_{H^k(\Omega_1)}^2
	\le r\|f\|_{H_{\mix}^{k,0}(\Omega_1\times\Omega_2)}^2
\]
and
\[	
	\sum_{\alpha=1}^r\lambda(\alpha)\|\psi(\alpha)\|_{H^k(\Omega_2)}^2
	\le r\|f\|_{H_{\mix}^{0,k}(\Omega_1\times\Omega_2)}^2.
\]
We will show later in Lemma~\ref{lem:regularity2} how to 
improve this estimate by sacrificing some regularity.

\subsection{Truncation error}\label{subsec:error}
We next give estimates on the decay rate of the 
eigenvalues of the integral operator $\mathcal{K}_f = 
\mathcal{S}_f^\star\mathcal{S}_f$ with kernel \eqref{eq:kernel}. 
To this end, we exploit the smoothness in the function's 
first variable and assume hence $f\in H_{\mix}^{k,0}(\Omega_1
\times\Omega_2)$. We introduce finite element spaces $U_r\subset 
L^2(\Omega_1)$, which consist of $r$ discontinuous, piecewise 
polynomial functions of total degree $\lceil k\rceil$ on a quasi-uniform 
triangulation of $\Omega_1$ with mesh width $h_r\sim r^{-1/n_1}$. 
Then, given a function \(w\in H^k(\Omega_1)\), the $L^2$-orthogonal 
projection $P_r:L^2(\Omega_1)\to U_r$ satisfies
\begin{equation}\label{eq:brambleh}
  \|(I-P_r)w \|_{L^2(\Omega_1)}
  	\le c_k r^{-k/n_1}|w|_{H^k(\Omega_1)}
\end{equation}
uniformly in \(r\) due to the Bramble-Hilbert lemma,
see e.g., \cite{B,Brenner}.

For the approximation of \(f(\boldsymbol{x},\boldsymbol{y})\) in the 
first variable, i.e. $\big((P_r\otimes I)f\big)(\boldsymbol{x},\boldsymbol{y})$,
we obtain the following approximation result for the present choice 
of \(U_r\), see \cite{GH18} for the proof.

\begin{lemma}\label{thm:optrerror}
Assume that $f\in H^k(\Omega_1\times\Omega_2)$ for some 
fixed $k\ge 0$. Let \(\lambda(1)\geq\lambda(2)\geq\ldots\geq 0\) be the 
eigenvalues of the operator \(\mathcal{K}_f=\mathcal{S}_f^\star\mathcal{S}_f\) 
and \(\lambda_r(1)\ge\lambda_r(2)\ge\ldots\ge\lambda_r(r)\geq 0\) 
those of \(\mathcal{K}_f^{\,r}:= P_r\mathcal{K}_fP_r\). Then, it holds 
\[
  \|f-(P_r\otimes I) f\|_{L^2(\Omega_1\times\Omega_2)}^2
  	= \trace\mathcal{K}_f-\trace\mathcal{K}_{f}^{\,r}
	= \sum_{\alpha=1}^{r}\big(\lambda(\alpha)-\lambda_r(\alpha)\big)
		+ \sum_{\alpha=r+1}^\infty\lambda(\alpha).
\]
\end{lemma}

By combining this lemma with the approximation estimate 
\eqref{eq:brambleh} and in view of $\lambda(\alpha)-\lambda_r(\alpha)
\ge 0$ for all $\alpha\in\{1,2,\ldots,r\}$ according to the min-max 
theorem of Courant-Fischer, see \cite{BO} for example, we 
conclude that the truncation error of the singular value 
decomposition can be bounded by
\begin{align*}
  \Bigg\|f-\sum_{\alpha=1}^r\sqrt{\lambda(\alpha)}
  	\big(\varphi(\alpha)\otimes\psi(\alpha)\big)\Bigg\|_{L^2(\Omega_1\times\Omega_2)}\!\!\!
  &= \sqrt{\sum_{\alpha=r+1}^\infty\lambda(\alpha)}
  \le c_k r^{-\frac{k}{n_1}}|f|_{H_{\mix}^{k,0}(\Omega_1\times\Omega_2)}.
\end{align*}

Since the eigenvalues of the integral operator $\mathcal{K}_f$ 
and its adjoint $\widetilde{\mathcal{K}}_f$ are the same, we 
can also exploit the smoothness of $f$ in the second coordinate
by interchanging the roles of $\Omega_1$ and $\Omega_2$ in 
the above considerations. We thus obtain the following theorem:

\begin{theorem}\label{coro:truncatedSVD}
Let $f\in H^k(\Omega_1\times\Omega_2)$ for some fixed $k\ge 0$
and let
\[
  f_r^{\SVD} = \sum_{\alpha=1}^r\sqrt{\lambda(\alpha)}
  	\big(\varphi(\alpha)\otimes\psi(\alpha)\big).
\]
Then, it holds
\begin{equation}\label{eq:truncation}
  \|f-f_r^{\SVD}\|_{L^2(\Omega_1\times\Omega_2)}
  = \sqrt{\sum_{\alpha=r+1}^\infty\lambda(\alpha)}
		\le c_k r^{-\frac{k}{\min\{n_1,n_2\}}}
			|f|_{H^k(\Omega_1\times\Omega_2)}.
\end{equation}
\end{theorem}


\begin{remark}
Theorem \ref{coro:truncatedSVD} implies that the eigenvalues
$\{\lambda(\alpha)\}_{\alpha\in\mathbb{N}}$ in case of a function
$f\in H^k(\Omega_1\times \Omega_2)$ decay like
\begin{equation}\label{eq:decay}
  \lambda(\alpha)\lesssim\alpha^{-\frac{2k}{\min\{n_1,n_2\}}-1}
  	\quad\text{as $\alpha\to\infty$}.
\end{equation}
\end{remark}

Having the decay rate of the eigenvalues at hand, we are 
able to improve the result of Lemma~\ref{lem:regularity1} 
by sacrificing some regularity. 
Note that the proof of this result is based upon an 
argument from \cite{S37}.

\begin{lemma}\label{lem:regularity2}
Assume that $f\in H^{k+\min\{n_1,n_2\}}(\Omega_1\times\Omega_2)$
for some fixed $k\ge 0$. Then, it holds
\[
  \sum_{\alpha=1}^\infty\lambda(\alpha)\|\varphi(\alpha)\|_{H^k(\Omega_1)}^2
  	=\|f\|_{H_{\mix}^{k,0}(\Omega_1\times\Omega_2)}^2
\]
and
\[
\sum_{\alpha=1}^\infty\lambda(\alpha)\|\psi(\alpha)\|_{H^k(\Omega_2)}^2
  	=\|f\|_{H_{\mix}^{0,k}(\Omega_1\times\Omega_2)}^2.
\]
\end{lemma}

\begin{proof}
Without loss of generality, we assume $n_1\le n_2$. Then,
since $f\in H^{k+n_1}(\Omega_1\times\Omega_2)$, we conclude 
from \eqref{eq:decay} that
\[
  \lambda(\alpha)\lesssim\alpha^{-\frac{2(k+n_1)}{n_1}-1} 
  	\quad\text{as $\alpha\to\infty$}.
\]
where we used that $n_1=\min\{n_1,n_2\}$. Moreover, by 
interpolating between $L^2(\Omega_1)$ and $H^{k+n_1}(\Omega_1)$, 
compare \cite{MCL,STEIN} for example, we find
\[
  \|\varphi(\alpha)\|_{H^k(\Omega_1)}^2
  	\lesssim\lambda(\alpha)^{-\frac{k}{k+n_1}},
\]
that is
\[
  \lambda(\alpha)\|\varphi(\alpha)\|_{H^k(\Omega_1)}^2
  	\lesssim\lambda(\alpha)^{\frac{n_1}{k+n_1}}.
\]
As a consequence, we infer that
\[
  \lambda(\alpha)\|\varphi(\alpha)\|_{H^k(\Omega_1)}^2
  	\lesssim \alpha^{-(\frac{2(k+n_1)}{n_1}+1)\cdot\frac{n_1}{k+n_1}}
  	= \alpha^{-(2+\delta)}
\]
with $\delta = \frac{n_1}{k+n_1}>0$. Therefore, it holds
\[
  \sum_{\alpha=1}^{\infty}\alpha^{(1+\delta')}\lambda(\alpha)\|\varphi(\alpha)\|_{H^k(\Omega_1)}^2 < \infty
\]
for any $\delta'\in(0,\delta)$. Hence, the series
\[
  A(\boldsymbol{x}) := \sum_{\alpha=1}^{\infty}\alpha^{(1+\delta')}\lambda(\alpha)
  	|\partial_{\boldsymbol{x}}^{\boldsymbol\beta}\varphi(\alpha,\boldsymbol{x})|^2
\]
converges for almost all $\boldsymbol{x}\in\Omega_1$, provided 
that $|\boldsymbol{\beta}|\le k$. Likewise, the series
\[
  B(\boldsymbol{y})  := \sum_{\alpha=1}^{\infty}\alpha^{-(1+\delta')}
  	|\psi(\alpha,\boldsymbol{y})|^2
\]
converges for almost all $\boldsymbol{y}\in\Omega_2$.
Thus, the series
\[
  \sum_{\alpha=1}^{\infty}\sqrt{\lambda(\alpha)}|\partial_{\boldsymbol{x}}^{\boldsymbol\beta}
  	\varphi(\alpha,\boldsymbol{x})||\psi(\alpha,\boldsymbol{y})|
		\le\sqrt{A(\boldsymbol{x})}\sqrt{B(\boldsymbol{y})}
\]
converges for almost all $\boldsymbol{x}\in\Omega_1$ and 
$\boldsymbol{y}\in\Omega_2$, provided that $|\boldsymbol{\beta}|\le k$.
Because of Egorov's theorem, the pointwise absolute convergence 
almost everywhere implies uniform convergence. Hence,
we can switch differentiation and summation to get
\begin{align*}
\|f\|_{H_{\mix}^{k,0}(\Omega_1\times\Omega_2)}^2
 &= \sum_{|\boldsymbol\beta|\le k}\Bigg\|\partial_{\boldsymbol{x}}^{\boldsymbol\beta}
 	\sum_{\alpha=1}^{\infty}\sqrt{\lambda(\alpha)}
		\big(\varphi(\alpha)\otimes\psi(\alpha)\big)\Bigg\|_{L^2(\Omega_1\times\Omega_2)}^2\\
 &= \sum_{|\boldsymbol\beta|\le k}\Bigg\|\sum_{\alpha=1}^{\infty}
 	\sqrt{\lambda(\alpha)}\big(\partial_{\boldsymbol{x}}^{\boldsymbol\beta}
  		\varphi(\alpha)\otimes\psi(\alpha)\big)\Bigg\|_{L^2(\Omega_1\times\Omega_2)}^2.
\end{align*}
Finally, we exploit the product structure of $L^2(\Omega_1\times\Omega_2)$ 
and the orthonormality of $\{\psi(\alpha)\}_{\alpha\in\mathbb{N}}$ to derive 
the first assertion, i.e., 
\begin{align*}
\|f\|_{H_{\mix}^{k,0}(\Omega_1\times\Omega_2)}^2
 &= \sum_{\alpha=1}^{\infty}\lambda(\alpha)\sum_{|\boldsymbol\beta|\le k}
 	\big\|\partial_{\boldsymbol{x}}^{\boldsymbol\beta}
  		\varphi(\alpha)\big\|_{L^2(\Omega_1)}^2\|\psi(\alpha)\|_{L^2(\Omega_2)}^2\\
 &= \sum_{\alpha=1}^{\infty}\lambda(\alpha)
 	\|\varphi(\alpha)\|_{H^k(\Omega_1)}^2.
\end{align*}
The second assertion follows in complete analogy.
\end{proof}


\subsection{Vector-valued functions}\label{subsec:vectors}
In addition to the aforementioned results, we will also need the
following result which concerns the approximation of vector-valued 
functions. Here and in the sequel, the vector-valued function 
$\boldsymbol{w} = [w(\alpha)]_{\alpha=1}^m$ is an element of 
$[L^2(\Omega)]^m$ and $[H^k(\Omega)]^m$ for some domain 
$\Omega\subset\mathbb{R}^n$, respectively, if the norms
\[
  \|\boldsymbol{w}\|_{[L^2(\Omega)]^m} 
  	= \sqrt{\sum_{\alpha=1}^m \|w(\alpha)\|_{L^2(\Omega)}^2},\quad
  \|\boldsymbol{w}\|_{[H^k(\Omega)]^m} 
  	= \sqrt{\sum_{\alpha=1}^m \|w(\alpha)\|_{H^k(\Omega)}^2}
\]
are finite. Likewise, the seminorm is defined in $[H^k(\Omega)]^m$.

Consider now a vector-valued function $\boldsymbol{w}\in [H^k(\Omega_1)]^m$ 
of dimension $m$. Then, instead of \eqref{eq:brambleh}, we find
\[
  \|(I-P_r)\boldsymbol{w}\|_{[L^2(\Omega_1)]^m}
  	\le c_k \bigg(\frac{r}{m}\bigg)^{-k/n_1}|\boldsymbol{w}|_{[H^k(\Omega_1)]^m},
\]
since $\boldsymbol{w}$ consists of $m$ components and we thus
need $m$-times as many ansatz functions for our approximation
argument. Hence, in case of a vector-valued function $\boldsymbol{f}\in 
[H_{\mix}^{k,0}(\Omega_1\times\Omega_2)]^m\simeq [H^k(\Omega_1)]^m
\otimes L^2(\Omega_2)$, we conclude by exploiting the smoothness in 
the first variable\footnote{Note that the kernel function of 
$\mathcal{S}_f^\star\mathcal{S}_f$ is matrix-valued while the kernel 
function of $\mathcal{S}_f\mathcal{S}_f^\star$ is scalar-valued.} that the truncation 
error of the singular value decomposition can be estimated by
\begin{equation}\label{eq:vectorSVD}
  \bigg\|\boldsymbol{f}-\sum_{\alpha=1}^r\sqrt{\lambda(\alpha)}
  	\big(\boldsymbol\varphi(\alpha)\otimes\psi(\alpha)\big)\Bigg\|_{[L^2(\Omega_1\times\Omega_2)]^m}\!\!\!
  \le c_k \bigg(\frac{r}{m}\bigg)^{-k/n_1}|\boldsymbol{f}|_{[H_{\mix}^{k,0}(\Omega_1\times\Omega_2)]^m}.
\end{equation}
Hence, the decay rate of the singular values is considerably reduced.
Finally, we like to remark that Lemma~\ref{lem:regularity2} holds 
also in the vector case, i.e.,
\[
  \sum_{\alpha=1}^\infty\lambda(\alpha)\|\boldsymbol\varphi(\alpha)\|_{[H^k(\Omega_1)]^m}^2
  	= \|\boldsymbol{f}\|_{[H_{\mix}^{k,0}(\Omega_1\times\Omega_2)]^m}^2
\]
and
\begin{equation}\label{eq:vectoresti}
  \sum_{\alpha=1}^\infty\lambda(\alpha)\|\psi(\alpha)\|_{H^k(\Omega_2)}^2
  	=\|\boldsymbol{f}\|_{[H_{\mix}^{0,k}(\Omega_1\times\Omega_2)]^m}^2,
\end{equation}
provided that $\boldsymbol{f}$ has extra regularity in terms of 
$\boldsymbol{f}\in [H^{k+n_1}(\Omega_1\times\Omega_2)]^m$.
Here, analogously to above, $[H_{\mix}^{0,k}(\Omega_1\times
\Omega_2)]^m\simeq [L^2(\Omega_1)]^m\otimes H^k(\Omega_2)$.

After these preparations, we now introduce and analyze two types 
of continuous analogues of tensor formats, namely of the Tucker format 
\cite{Hitch,Tucker} and of the tensor train format \cite{MPS,O}, and 
discuss their approximation properties for functions $f\in 
H^k(\Omega_1\times\dots\times\Omega_m)$.

\section{Tucker tensor format}\label{sec:TF}
\subsection{Tucker decompostion}
We shall consider from now on a product domain which consists of 
$m$ different domains $\Omega_j\subset\mathbb{R}^{n_j}$, $j=1,\ldots,m$. 
For given $f\in L^2(\Omega_1\times\dots\times\Omega_m)$ and
$j\in\{1,2,\ldots,m\}$, we apply the singular value decomposition to separate 
the variables $\boldsymbol{x}_j\in\Omega_j$ and $(\boldsymbol{x}_1,\ldots,
\boldsymbol{x}_{j-1},\boldsymbol{x}_{j+1},\ldots,\boldsymbol{x}_m)\in\Omega_1
\times\dots\times\Omega_{j-1}\times\Omega_{j+1}\times\dots\times\Omega_m$. 
We hence get
\begin{equation}\label{eq:tuckerSVD}
\begin{aligned}
  &f(\boldsymbol{x}_1,\ldots,\boldsymbol{x}_{j-1},\boldsymbol{x}_j,\boldsymbol{x}_{j+1},\ldots,\boldsymbol{x}_m)\\
  &\qquad= \sum_{\alpha_j=1}^\infty\sqrt{\lambda_j(\alpha_j)}\varphi_j(\boldsymbol{x}_j,\alpha_j)
		\psi_j(\alpha_j,\boldsymbol{x}_1,\ldots,\boldsymbol{x}_{j-1},\boldsymbol{x}_{j+1},\ldots,\boldsymbol{x}_m),
\end{aligned}
\end{equation}
where the left eigenfunctions $\{\varphi_j(\alpha_j)\}_{\alpha_j\in\mathbb{N}}$
form an orthonormal basis in $L^2(\Omega_j)$. Consequently, if we iterate over all
$j\in\{1,2,\ldots,m\}$, this yields an orthonormal basis $\{\varphi_1(\alpha_1)\otimes
\cdots\otimes\varphi_m(\alpha_m)\}_{\boldsymbol\alpha\in\mathbb{N}^m}$ of 
$L^2(\Omega_1\times\dots\times\Omega_m)$, and we arrive at the representation
\begin{equation}\label{eq:core}
  f(\boldsymbol{x}_1,\ldots,\boldsymbol{x}_m) = \sum_{|\boldsymbol\alpha| = 1}^\infty 
  	\omega(\boldsymbol\alpha)\varphi_1(\alpha_1,\boldsymbol{x}_1)\cdots\varphi_m(\alpha_m,\boldsymbol{x}_m).
\end{equation}
Herein, the tensor $\big[\omega(\boldsymbol\alpha)\big]_{\boldsymbol
\alpha\in\mathbb{N}^m}$ is the \emph{core tensor}, where a single coefficient 
is given by
\[
  \omega(\alpha_1,\ldots,\alpha_m) = \int_{\Omega_1\times\dots\times\Omega_m} 
  	f(\boldsymbol{x}_1,\ldots,\boldsymbol{x}_m) \varphi_1(\alpha_1,\boldsymbol{x}_1)
		\cdots\varphi_m(\alpha_m,\boldsymbol{x}_m)\d\,(\boldsymbol{x}_1,\ldots,\boldsymbol{x}_m).
\]

\subsection{Truncation error}
If we intend to truncate the singular value decomposition 
\eqref{eq:tuckerSVD} after $r_j$ terms such that the truncation 
error is bounded by $\varepsilon$, we have to choose 
\begin{equation}\label{eq:ranksTF}
  \sqrt{\sum_{\alpha_j = r_j+1}^\infty\lambda_j(\alpha_j)}
  	\lesssim r_j^{-k/n_j}|f|_{H^k(\Omega_1\times\dots\times\Omega_m)}
	\overset{!}{\lesssim}\varepsilon
  	\qquad\Longrightarrow\qquad r_j = \varepsilon^{-{n_j}/k}
\end{equation}
according to Theorem \ref{coro:truncatedSVD}. Doing so 
for all $j\in\{1,2,\ldots,m\}$, we obtain the approximation
\[
  f_{r_1,\ldots,r_m}^{\TF}(\boldsymbol{x}_1,\ldots,\boldsymbol{x}_m) 
  	= \sum_{\alpha_1=1}^{r_1}\cdots\sum_{\alpha_m=1}^{r_m}
  		\omega(\alpha_1,\ldots,\alpha_m)\varphi_1(\alpha_1,\boldsymbol{x}_1)
			\cdots\varphi_m(\alpha_m,\boldsymbol{x}_m).
\]
We have the following result on the Tucker decomposition:

\begin{theorem}
Let $f\in H^k(\Omega_1\times\dots\times\Omega_m)$ for 
some fixed $k>0$ and $0<\varepsilon<1$. If the ranks are chosen 
according to $r_j=\varepsilon^{-n_j/k}$ for all $j=1,\dots,m$. Then, 
the truncation error of the truncated Tucker decomposition is
\[
 \big\|f-f_{r_1,\ldots,r_m}^{\TF}\big\|_{L^2(\Omega_1\times\dots\times\Omega_m)}
 \le \sqrt{\sum_{j=1}^m \sum_{\alpha_j = r_j+1}^\infty\lambda_j(\alpha_j)}\lesssim\sqrt{m}\varepsilon,
\]
while the storage cost for the core tensor
of $f_{r_1,\ldots,r_m}^{\TF}$ are $\prod_{j=1}^m r_j = 
\varepsilon^{-(n_1+\dots+n_m)/k}$.
\end{theorem}

\begin{proof}
For the approximation of the core tensor, the sets of the univariate 
eigenfunctions $\{\varphi_j(\alpha_j)\}_{\alpha_j=1}^{r_j}$ are used 
for all $j=1,\ldots,m$, cf.~\eqref{eq:core}. Due to orthonormality, we find
\[
  \big\|f-f_{r_1,\ldots,r_m}^{\TF}\big\|_{L^2(\Omega_1\times\dots\times\Omega_m)}^2
  	= \sum_{j=1}^m \big\|f_{r_1,\dots,r_{j-1},\infty,\dots,\infty}^{\TF}
  	-f_{r_1,\dots,r_j,\infty,\dots,\infty}^{\TF}\big\|_{L^2(\Omega_1\times\dots\times\Omega_m)}^2,
\]
where we obtain $f_{\infty,\ldots,\infty}^{\TF} = f$ in 
case of $j=1$. Since
\begin{align*}
  &\big\|f_{r_1,\dots,r_{j-1},\infty,\dots,\infty}^{\TF}
  	-f_{r_1,\dots,r_j,\infty,\dots,\infty}^{\TF}\big\|_{L^2(\Omega_1\times\dots\times\Omega_m)}^2\\
 &\qquad\qquad\le \big\|f_{\infty,\dots,\infty}^{\TF}
  	-f_{\infty,\dots,\infty,r_j,\infty,\dots,\infty}^{\TF}\big\|_{L^2(\Omega_1\times\dots\times\Omega_m)}^2
  = \sum_{\alpha_j = r_j+1}^\infty\lambda_j(\alpha_j)
\end{align*}
for all $j\in\{1,2,\ldots,m\}$, we arrive with \eqref{eq:ranksTF}
and the summation over $j = 1,\ldots,m$ at the desired error estimate.
This completes the proof, since the estimate on the 
number of coefficients in the core tensor is obvious.
\end{proof}

\subsection{Sobolev spaces with dimension weights}\label{sct:Tucker:infty}
The cost of the core tensor of the Tucker decomposition exhibit 
the curse of dimension as the number $m$ of subdomains increases.
This can be seen most simply for the example $n_j=n$. Then, the 
cost are $\varepsilon^{-nm/k}$, which expresses the curse of dimension 
as long as $k$ is not proportional to $m$. Nonetheless, in case of 
Sobolev spaces with dimension weights, the curse of dimension can be beaten. 

For $f\in H^{k+n}(\Omega^m)$, we shall discuss the situation $m\to\infty$ 
in more detail. To this end, we assume that all subdomains are identical 
to a single domain $\Omega\subset\mathbb{R}^n$ of dimension $n$
and note that the limit $m\to\infty$ only makes sense when weights 
are included in the underlying Sobolev spaces which ensure 
that higher dimensions become less important. For our proofs
we choose as usual $m$ arbitrary but fixed and show the existence of 
$m$-independent constants in the convergence and cost estimates.

The Sobolev spaces $H_{\boldsymbol\gamma}^k(\Omega^m)$ 
with dimension weights $\boldsymbol\gamma\in\mathbb{R}^m$ we consider are 
given by all functions $f\in H^k(\Omega^m)$ such that
\begin{equation}\label{eq:weights}
  \bigg\|\frac{\partial^k f}{\partial\boldsymbol{x}_j^{\boldsymbol\beta}}\bigg\|_{L^2(\Omega^m)}
	\lesssim\gamma_j^k\|f\|_{H^k(\Omega^m)}\ \text{for all}\ |\boldsymbol\beta| = k\ 
		\text{and}\ j=1,2,\ldots,m.
\end{equation}
The definition in \eqref{eq:weights} means that, given 
a function $f$ with norm $\|f\|_{H^k(\Omega^m)}<\infty$, 
the partial derivatives with respect to $\boldsymbol{x}_j$ become
less important as the dimension $j$ increases. Such functions
appear for example in uncertainty quantification. Let be given a 
Karhunen-Lo\`eve expansion 
\[
  u({\bf x},{\bf y}) = \sum_{j = 1}^m \sigma_j \varphi_j({\bf x}) y_j,
  \quad y_j\in [-\nicefrac{1}{2},\nicefrac{1}{2}],
\]
and insert it into a function $b:\mathbb{R}\to\mathbb{R}$
of \emph{finite} smoothness $W^{k,\infty}(\mathbb{R})$. 
Then, the function $b\big(u({\bf x},{\bf y})\big)$ satisfies 
\eqref{eq:weights} with respect to the ${\bf y}$-variable, 
where $\gamma_j = \sigma_j$. Hence, the solution of 
a given partial differential equation would satisfy a 
decay estimate similar to \eqref{eq:weights} whenever 
the stochastic field enters the partial differential equation 
through a non-smooth coefficient function $b$,
compare \cite{GP18,HPS16,HS20} for example.

It turns out that algebraically decaying weights 
\eqref{eq:weightsTF} are sufficient to beat the curse of 
dimension in case of the Tucker tensor decomposition.%
\footnote{In Theorem \ref{thm:weightedTF}, no truncation 
of the dimension is applied, as it would be required in practice 
if the number $m$ of domains tends to infinity. Note that the 
dimension truncation is indeed here the same as for the tensor 
train decomposition later on, see also Theorem~\ref{thm:wheightTT}.}

\begin{theorem}\label{thm:weightedTF}
Given $\delta > 0$, let $f\in H_{\boldsymbol\gamma}^k(\Omega^m)$ 
for some fixed $k>0$ with weights \eqref{eq:weights} 
that decay like
\begin{equation}\label{eq:weightsTF}
  \gamma_j\lesssim j^{-(1+\delta')/k}\ \text{for some $\delta'>\delta+\frac{k}{n}$}.
\end{equation}
Then, for all $0<\varepsilon<1$, the error of the continuous 
Tucker decomposition with ranks
\begin{equation}\label{eq:ranksTFweighted}
  r_j = \big\lceil\gamma_j^n j^{(1+\delta)n/k}\varepsilon^{-n/k}\big\rceil
\end{equation}
is of order $\varepsilon$ while the storage cost for the 
core tensor of $f_{r_1,\ldots,r_m}^{\TF}$ are bounded by
$\varepsilon^{-n/k}$ independent of the dimension $m$.
\end{theorem}

\begin{proof}
In view of Theorem~\ref{coro:truncatedSVD} and \eqref{eq:weights}, 
we deduce by choosing the ranks as in \eqref{eq:ranksTF} that
\[
  \sqrt{\sum_{\alpha_j = r_j+1}^\infty\lambda_j(\alpha_j)}
  	\lesssim r_j^{-k/n} \gamma_j^k \|f\|_{H^k(\Omega^m)}
		\lesssim\frac{\varepsilon}{j^{1+\delta}}.
\]
Therefore, we reach the desired over-all truncation error 
\begin{equation}\label{eq:over-error}
  \sum_{j=1}^m \frac{\varepsilon}{j^{1+\delta}} \lesssim \varepsilon\ \text{as $m\to\infty$}.
\end{equation}
When the weights $\gamma_j$ decay as in \eqref{eq:ranksTFweighted}, 
then the cost of the core tensor are
\[
  C := \prod_{j=1}^m r_j \le\prod_{j=1}^m \big(1+\gamma_j^n j^{(1+\delta)n/k}\varepsilon^{-n/k}\big)
  \lesssim\prod_{j=1}^m \big(1+j^{-\theta}\varepsilon^{-n/k}\big)
\]
with $\theta = (\delta'-\delta)n/k>1$. Hence, the cost of the 
core tensor stay bounded independently of $m$ since
\[
  \log C\lesssim\sum_{j=1}^m \log\big(1+j^{-\theta}\varepsilon^{-n/k}\big)
  \le \varepsilon^{-n/k}\sum_{j=1}^m j^{-\theta}\lesssim\varepsilon^{-n/k}
  	\ \text{as $m\to\infty$}.
\]
\end{proof}

\section{Tensor train format}\label{sec:TT}
\subsection{Tensor train decomposition}
For the discussion of the continuous tensor train decomposition, we 
should assume that the domains $\Omega_j\subset\mathbb{R}^{n_j}$, 
$j=1,\ldots,m$, are arranged in such a way that it holds $n_1\le\dots\le 
n_m$.\footnote{The 
considerations in this section are based upon \cite{M16}. Nonetheless, 
the results derived there are not correct. The authors did not consider 
the impact of the vector-valued singular value decomposition in a proper
way, which indeed does result in the curse of dimension.}

Now, consider $f\in H^k(\Omega_1\times\dots\times\Omega_m)$ and 
separate the variables $\boldsymbol{x}_1\in\Omega_1$ and 
$(\boldsymbol{x}_2,\ldots,\boldsymbol{x}_m)\in\Omega_2\times\dots
\times\Omega_m$ by the singular value decomposition
\[
  f(\boldsymbol{x}_1,\boldsymbol{x}_2,\ldots,\boldsymbol{x}_n)
  	= \sum_{\alpha_1=1}^\infty\sqrt{\lambda_1(\alpha_1)}\varphi_1(\boldsymbol{x}_1,\alpha_1)
		\psi_1(\alpha_1,\boldsymbol{x}_2,\ldots,\boldsymbol{x}_m).
\]
Since
\[
   \bigg[\sqrt{\lambda_1(\alpha_1)}\psi_1(\alpha_1)\bigg]_{\alpha_1=1}^\infty
   \in \ell^2(\mathbb{N})\otimes L^2(\Omega_2\times\dots\times\Omega_m),
\]
we can separate $(\alpha_1,\boldsymbol{x}_2)\in\mathbb{N}
\times\Omega_2$ from $(\boldsymbol{x}_3,\ldots,\boldsymbol{x}_m)\in
\Omega_3\times\dots\times\Omega_m$ by means of a second singular 
value decomposition and arrive at
\begin{equation}\label{eq:tt-step}
\begin{aligned}
 &\bigg[\sqrt{\lambda_1(\alpha_1)}\psi_1(\alpha_1,\boldsymbol{x}_2,
 		\ldots,\boldsymbol{x}_m)\bigg]_{\alpha_1=1}^\infty\\
  	&\qquad\qquad= \sum_{\alpha_2=1}^\infty\sqrt{\lambda_2(\alpha_2)}
		\bigg[\varphi_2(\alpha_1,\boldsymbol{x}_2,\alpha_2)\bigg]_{\alpha_1=1}^\infty
			\psi_2(\alpha_2,\boldsymbol{x}_3,\ldots,\boldsymbol{x}_m).
\end{aligned}
\end{equation}
By repeating the last step and successively separating
$(\alpha_{j-1},\boldsymbol{x}_j)\in\mathbb{N}\times\Omega_j$ from 
$(\boldsymbol{x}_{j+1},\ldots,\boldsymbol{x}_m)\in\Omega_{j+1}
\times\dots\times\Omega_m$ for $j=3,\ldots,m-1$ we finally 
arrive at the representation
\begin{align*}
  f(\boldsymbol{x}_1,\ldots,\boldsymbol{x}_m) &= \sum_{\alpha_1=1}^\infty\cdots\sum_{\alpha_{m-1}=1}^\infty
  	\varphi_1(\alpha_1,\boldsymbol{x}_1)\varphi_2(\alpha_1,\boldsymbol{x}_2,\alpha_2)\\
  &\hspace*{2.5cm}\cdots\varphi_{m-1}(\alpha_{m-2},\boldsymbol{x}_{m-1},\alpha_{m-1})
	\varphi_m(\alpha_{m-1},\boldsymbol{x}_m),
\end{align*}
where 
\[
  \varphi_m(\alpha_{m-1},\boldsymbol{x}_m) 
  	= \sqrt{\lambda_{m-1}(\alpha_{m-1})}\psi_{m-1}(\alpha_{m-1},\boldsymbol{x}_m).
\]
In contrast to the Tucker format, we do not obtain a huge core tensor
since each of the $m-1$ singular value decompositions of the tensor 
train decomposition removes the actual first spatial domain from the
approximant. We just obtain a product of \emph{matrix\/}-valued functions
(except for the first and last factor which are vector-valued functions),
each of which is related with a specific domain $\Omega_j$. This 
especially results in only $m-1$ sums in contrast to the 
$m$ sums for the Tucker format.

\subsection{Truncation error}
In practice, we truncate the singular value decomposition in step $j$ 
after $r_j$ terms, thus arriving at the representation
\begin{align*}
  f_{r_1,\ldots,r_{m-1}}^{\TT}(\boldsymbol{x}_1,\ldots,\boldsymbol{x}_m)
  	&= \sum_{\alpha_1=1}^{r_1}\cdots\sum_{\alpha_{m-1}=1}^{r_{m-1}}\varphi_1(\alpha_1,\boldsymbol{x}_1)
		\varphi_2(\alpha_1,\boldsymbol{x}_2,\alpha_2)\\&\hspace*{3cm}\cdots
		\varphi_{m-1}(\alpha_{m-2},\boldsymbol{x}_{m-1},\alpha_{m-1})
			\varphi_m(\alpha_{m-1},\boldsymbol{x}_m).
\end{align*}
One readily infers by using again Pythogoras' theorem that
the truncation error is bounded by
\[
  \|f-f_{r_1,\ldots,r_{m-1}}^{\TT}\|_{L^2(\Omega_1\times\dots\times\Omega_m)}
  	\le\sqrt{\sum_{j=1}^{m-1}\sum_{\alpha_j=r_j+1}^\infty \lambda_j(\alpha_j)},
\]
see also \cite{OT}. Note that, for $j\ge 2$, the singular 
values $\{\lambda_j(\alpha)\}_{\alpha\in\mathbb{N}}$ in this 
estimate do not coincide with the singular values from the 
original continuous tensor train decomposition due to
the truncation.

We next shall give bounds on the truncation error. In the 
$j$-th step of the algorithm, $j=2,3,\ldots,m-1$, one needs to 
approximate the vector-valued function
\[
  \boldsymbol{g}_j(\boldsymbol{x}_j,\ldots,\boldsymbol{x}_m) 
  	:= \bigg[\sqrt{\lambda_{j-1}(\alpha_{j-1})}
		\psi_{j-1}(\alpha_{j-1},\boldsymbol{x}_j,
			\ldots,\boldsymbol{x}_m)\bigg]_{\alpha_{j-1}=1}^{r_{j-1}}
\]
by a vector-valued singular value decomposition.
This means that we consider the singular value decomposition
\eqref{eq:vectorSVD} for a vector-valued function in case of the
domains $\Omega_j$ and $\Omega_{j+1}\times\dots\times\Omega_m$.

For $f\in H^{k+n_{m-1}}(\Omega_1\times\dots\times\Omega_m)$, it holds 
$\boldsymbol{g}_2\in [H^{k+n_{m-1}}(\Omega_2\times\dots\times\Omega_m)]^{r_1}$ 
and 
\[
  |\boldsymbol{g}_2|_{[H^k(\Omega_2\times\dots\times\Omega_m)]^{r_1}}
  \le\sqrt{\sum_{\alpha_1=1}^\infty\lambda_1(\alpha_1)
		|\psi_1(\alpha_1)|_{H^k(\Omega_2\times\dots\times\Omega_m)}^2}
  	\le|f|_{H^k(\Omega_1\times\dots\times\Omega_m)}
\]
according to Lemma~\ref{lem:regularity2}, precisely in its 
vectorized version \eqref{eq:vectoresti}. It follows $\boldsymbol{g}_3
\in [H^{k+n_{m-1}}(\Omega_3\times\dots\times\Omega_m)]^{r_2}$ and, 
again by \eqref{eq:vectoresti},
\[
  |\boldsymbol{g}_3|_{[H^k(\Omega_3\times\dots\times\Omega_m)]^{r_2}}
  	\le\sqrt{\sum_{\alpha_2=1}^\infty\lambda_2(\alpha_2)
		|\psi_2(\alpha_2)|_{H^k(\Omega_3\times\dots\times\Omega_m)}^2}
  	\le|\boldsymbol{g}_2|_{[H^k(\Omega_2\times\dots\times\Omega_m)]^{r_1}}.
\]
We hence conclude recursively $\boldsymbol{g}_j\in 
[H^{k+n_{m-1}}(\Omega_j\times\dots\times\Omega_m)]^{r_{j-1}}$ and 
\begin{equation}\label{eq:boundedness}
  |\boldsymbol{g}_j|_{[H^k(\Omega_j\times\dots\times\Omega_m)]^{r_{j-1}}}
  	\le|f|_{H^k(\Omega_1\times\dots\times\Omega_m)}\ \text{for all}\ j=2,3,\ldots,m-1.
\end{equation}

Estimate \eqref{eq:boundedness} shows that the $H^k$-seminorm 
of the vector-valued functions $\boldsymbol{g}_j$ stays bounded by 
$|f|_{H^k(\Omega_1\times\dots\times\Omega_m)}$. But according to 
\eqref{eq:vectorSVD}, we have in the $j$-th step only the truncation error
estimate
\begin{align*}
  &\Bigg\|\boldsymbol{g}_j-\sum_{\alpha_j=1}^{r_j}\sqrt{\lambda_j(\alpha_j)}
  	\big(\boldsymbol\varphi_j(\alpha_j)\otimes\psi_j(\alpha_j)\big)
		\Bigg\|_{[L^2(\Omega_j\times\dots\times\Omega_m)]^{r_{j-1}}}\\
			&\hspace*{5cm}\lesssim \bigg(\frac{r_j}{r_{j-1}}\bigg)^{-k/n_j}
				|\boldsymbol{g}_j|_{[H^k(\Omega_j\times\dots\times\Omega_m)]^{r_{j-1}}}.
\end{align*}
Hence, in view of \eqref{eq:boundedness}, to achieve the target 
accuracy $\varepsilon$ per truncation, the truncation ranks need 
to be increased in accordance with
\begin{equation}\label{eq:ranksTT}
  r_1 = \varepsilon^{-n_1/k},\quad r_2 = \varepsilon^{-(n_1+n_2)/k},\quad 
  	\ldots, \quad r_{m-1} = \varepsilon^{-(n_1+\dots+n_{m-1})/k}.
\end{equation}

We summarize our findings in the following theorem, which
holds in this form also if the subdomains $\Omega_j\subset
\mathbb{R}_{n_j}$ are not ordered in such a way that $n_1\le\dots\le n_m$.

\begin{theorem}
Let $f\in H^{k+\max\{n_1,\ldots,n_m\}}(\Omega_1\times\dots\times\Omega_m)$
for some fixed $k>0$ and $0<\varepsilon<1$. Then, the over-all truncation error 
of the tensor train decomposition with truncation ranks \eqref{eq:ranksTT} is
\[
  \|f-f_{r_1,\ldots,r_{m-1}}^{\TT}\|_{L^2(\Omega_1\times\dots\times\Omega_m)}
  	\lesssim\sqrt{m}\varepsilon.
\]
The storage cost for $f_{r_1,\ldots,r_{m-1}}^{\TT}$
are given by
\begin{equation}\label{eq:costTT}
  r_1 + \sum_{j=2}^{m-1} {r_{j-1} r_j}
  	= \varepsilon^{-n_1/k} + \varepsilon^{-(2n_1+n_2)/k} 
		+\dots+ \varepsilon^{-(2n_1+\dots+2n_{m-2}+n_{m-1})/k}
\end{equation}
and hence are bounded by $\mathcal{O}(\varepsilon^{-(2m-1)
\max\{n_1,\ldots,n_{m-1}\}/k})$.
\end{theorem}

\begin{remark}
If $n:=n_1=\dots=n_m$, then the cost of the tensor train decomposition 
are $\mathcal{O}(\varepsilon^{-(2m-1)n/k})$. Thus, the cost are quadratic 
compared to the cost of the Tucker decomposition. However, in practice, 
one performs $m/2$ forward steps and $m/2$ backward steps. This 
means one computes $m/2$ steps as described above to successively 
separate $\boldsymbol{x}_1,\boldsymbol{x}_2,\ldots,\boldsymbol{x}_{m/2}$ from 
the other variables. Then, one performs the algorithm in the opposite direction, 
i.e., one successively separates $\boldsymbol{x}_{m},\boldsymbol{x}_{m-1},
\ldots,\boldsymbol{x}_{m/2+1}$ from the other variables. This way, the over-all 
cost are reduced to the order $\mathcal{O}(\varepsilon^{-mn/k})$.\footnote{If 
the spatial dimensions $n_j$, $j=1,\ldots,m$, of the subdomains are 
different, one can balance the number of forward and backward 
steps in a better way to reduce the cost further.}
\end{remark}

\subsection{Sobolev spaces with dimension weights}
Like for the Tucker decomposition, the cost of the tensor 
train decomposition suffer from the curse of dimension as 
the number $m$ of subdomains increases. We therefore
discuss again appropriately Sobelev spaces with dimension weights,  
where we assume for reasons of simplicity that all subdomains are 
identical to a single domain $\Omega\subset\mathbb{R}^n$
of dimension $n$.

\begin{theorem}\label{thm:wheightTT}
Given $\delta > 0$, let $f\in H_{\boldsymbol\gamma}^{k+n}(\Omega^m)$ 
for some fixed $k>0$ with weights \eqref{eq:weights} that decay 
like \eqref{eq:weightsTF}. For $0<\varepsilon<1$, choose 
the ranks successively in accordance with
\begin{equation}\label{eq:ranksTTweighted}
  r_j = \big\lceil r_{j-1}\gamma_j^n j^{(1+\delta)n/k}\varepsilon^{-n/k}\big\rceil
\end{equation}
if $j\le M$ and $r_j = 0$ if $j>M$. Here, $M$ is given by
\begin{equation}\label{eq:truncTT}
  M = \varepsilon^{-1/(1+\delta')}.
\end{equation}
Then, the error of the continuous tensor train decomposition is 
of order $\varepsilon$ while the storage cost of $f_{r_1,\ldots,r_m}^{\TT}$ 
stay bounded by $M\exp(\varepsilon^{-n/k})^2$ independent of the 
dimension $m$.
\end{theorem}

\begin{proof}
The combination of Theorem~\ref{coro:truncatedSVD}, \eqref{eq:weights} 
and \eqref{eq:ranksTTweighted} implies 
\[
  \sqrt{\sum_{\alpha_j = r_j+1}^\infty\lambda_j(\alpha_j)}
  	\lesssim \bigg(\frac{r_j}{r_{j-1}}\bigg)^{-k/n} \gamma_j^k 
		\|f\|_{H^k(\Omega^m)}\lesssim\frac{\varepsilon}{j^{1+\delta}},
			\quad j=1,2,\ldots,M,
\]
and 
\[
  \sqrt{\sum_{\alpha_{M+1} = 1}^\infty\lambda_{M+1}(\alpha_{M+1})}\lesssim 
  	\gamma_{M+1}^k\|f\|_{H^k(\Omega^m)}\lesssim\varepsilon\|f\|_{H^k(\Omega^m)}.
\]
Hence, as in the proof of Theorem \ref{thm:weightedTF}, the 
approximation error of the continuous tensor train decomposition is
bounded by a multiple of $\varepsilon$ independent of $m$.

Next, we observe for all $j\le M$ that
\[
  r_j \lesssim \big\lceil r_{j-1}\gamma_j^n j^{(1+\delta)n/k}\varepsilon^{-n/k}\big\rceil
  	\lesssim r_{j-1}\gamma_j^n j^{(1+\delta)n/k}\varepsilon^{-n/k}+1.
\]
This recursively yields
\begin{align*}
  r_j-1&\lesssim \sum_{p=1}^j \prod_{q=p}^j \gamma_q^n q^{(1+\delta)n/k}\varepsilon^{-n/k}\\
  	&= \sum_{p=1}^j\varepsilon^{(p-j-1)n/k}\prod_{q=p}^j q^{-\theta}\\
	&= \sum_{p=1}^j \varepsilon^{(p-j-1)n/k}\bigg(\frac{(p-1)!}{j!}\bigg)^\theta\\
	&= \sum_{p=1}^j \varepsilon^{-p n/k}\bigg(\frac{(j-p)!}{j!}\bigg)^\theta.
\end{align*}
Hence, by using that $\theta = (\delta'-\delta)n/k>1$, we obtain
\[
 r_j\lesssim \sum_{p=0}^j \varepsilon^{-p n/k}\frac{(j-p)!}{j!}
 	\le\sum_{p=0}^j \frac{\varepsilon^{-p n/k}}{p!}\le\exp(\varepsilon^{-n/k}).
\]
Therefore, the cost \eqref{eq:costTT} are 
\[
  r_1 + \sum_{j=2}^M r_{j-1} r_j\le\sum_{j=1}^M r_j^2 
  \lesssim M \exp(\varepsilon^{-n/k})^2
\]
and, hence, are bounded independently of $m$ in view of \eqref{eq:truncTT}.
\end{proof}

\section{Discussion and conclusion}\label{sec:conrem}
In the present article, we considered the continuous 
versions of the Tucker tensor format and of the tensor train 
format for the approximation of functions which live on an
$m$-fold product of arbitrary subdomains. By considering
(isotropic) Sobolev smoothness, we derived estimates on 
the ranks to be chosen in order to realize a prescribed target 
accuracy. These estimates exhibit the curse of dimension.

Both tensor formats have in common that 
always only the variable with respect to a single domain is 
separated from the other variables by means of the singular 
value decomposition. This enables cheaper storage schemes, 
while the influence of the over-all dimension of the product 
domain is reduced to a minimum. 

We also examined the situation of Sobolev spaces with dimension weights.
Having sufficiently fast decaying weights helps to beat the curse 
of dimension as the number of subdomains tends to infinity. It turned 
out that algebraically decaying weights are appropriate for both, the 
Tucker tensor format and the tensor train format.


We finally remark that we considered here only the ranks of the
tensor decomposition in the \emph{continuous\/} case, i.e., for
functions and not for tensors of discrete data. Of course, an 
additional projection step onto suitable finite dimensional trial 
spaces on the individual domains would be necessary to arrive 
at a fully discrete approximation scheme that can really be 
used in computer simulations. This would impose a further error 
of discretization type which needs to be balanced with the 
truncation error of the particular continuous tensor format.

\subsection*{Acknowledgement}
Michael Griebel was partially supported by the Sonder\-for\-schungs\-bereich 
1060 {\em The Mathematics of Emergent Effects} funded by the 
Deutsche For\-schungs\-gemeinschaft. Both authors like to thank 
Reinhold Schneider (Technische Universit\"at Berlin) very much 
for fruitful discussions about tensor approximation.


\end{document}